\documentclass[11pt,a4paper]{article}

\usepackage[english]{babel}
\usepackage{amsthm}
\usepackage{amsmath,amssymb}
\usepackage{graphicx}
\usepackage[small,bf,hang]{caption} 

\usepackage{hyperref}
\usepackage{enumerate} 

\usepackage{color}

\newtheorem{theorem}{Theorem}[section]
\theoremstyle{plain}

\newtheorem{lemma}[theorem]{Lemma}

\newtheorem{obs}[theorem]{Observation}
\theoremstyle{definition}

\graphicspath{{figures/}}

\begin{document}


\title{The Generation of Fullerenes}

\author{Gunnar Brinkmann\\ \footnote{Gunnar.Brinkmann@UGent.be}
Applied Mathematics \& Computer Science\\
  Ghent University\\
Krijgslaan 281-S9, \\9000 Ghent, Belgium\\
\and
Jan Goedgebeur\\
\footnote{Jan.Goedgebeur@UGent.be}
Applied Mathematics \& Computer Science\\
  Ghent University\\
Krijgslaan 281-S9, \\9000 Ghent, Belgium\\
\and
Brendan D. McKay\\
\footnote{bdm@cs.anu.edu.au}
Research School of Computer Science\\  
Australian National University\\  
ACT 0200, Australia 
}

\date{} 
\maketitle

\begin{abstract}

We describe an efficient new algorithm for the generation of fullerenes. 
Our implementation of this algorithm is
more than $3.5$ times faster than the previously fastest generator
for fullerenes -- {\em fullgen} --  
and the first program since {\em fullgen} to be useful for more 
than 100 vertices.
We also note a programming error in {\em fullgen} that caused problems for 136 or more vertices.
We tabulate the numbers of fullerenes and IPR fullerenes up to 400 vertices.
We also check up to 316 vertices
a conjecture of Barnette that cubic planar graphs with maximum face
size 6 are hamiltonian and verify that the smallest counterexample to the spiral conjecture 
has 380 vertices. \\

Note: this is the unedited version of our paper which was submitted and subsequently accepted for publication in \textit{Journal of Chemical Information and Modeling}. The final edited and published version can be accessed at \url{http://dx.doi.org/10.1021/ci3003107}
\end{abstract}

\section{Introduction}
Fullerenes are spherical carbon molecules that can be modelled
as 
cubic plane graphs where all faces are pentagons or
hexagons. We will refer to these mathematical models also
as fullerenes. Euler's formula implies that a
fullerene with $n$ vertices contains exactly 12 pentagons and $n/2 - 10$ hexagons. 
The
\textit{dual} of a fullerene is the plane graph
obtained by exchanging the roles of vertices and faces: the vertex set of the dual graph
is the set of faces of the original graph and two 
vertices in the dual graph are adjacent if and only if the two faces share an edge
in the original graph. The rotational order around the vertices in
the embedding
of the dual fullerene  follows the rotational order of the faces.
As fullerenes and their duals are 3-connected, due to the theorem of Whitney the plane embeddings of fullerenes
and duals of fullerenes are uniquely determined and the concept of graph isomorphism
and isomorphism of embedded graphs (treating the mirror image as equivalent)
coincide. The dual
of a fullerene with $n$ vertices is a triangulation (i.e. every face is a
triangle) which contains 12 vertices with degree 5 and $n/2
- 10$ vertices with degree 6.

Isolated Pentagon Rule (IPR) fullerenes are fullerenes where no two
pentagons share an edge. IPR fullerenes are especially interesting 
due to a general tendency to be chemically more stable and thus more likely to occur
in nature.

The first fullerene molecule was discovered in 1985 by Kroto et
al.~\cite{kroto_85}, namely the famous $C_{60}$
buckminsterfullerene or ``buckyball''. After that discovery several
attempts have been made to generate complete lists of fullerene isomers.

The first approach was the spiral algorithm given by
Manolopoulos et al. in 1991~\cite{manolopoulos_91}. This algorithm
was relatively inefficient and also
incomplete in the sense that not every fullerene isomer could be
generated with it. Manolopoulos and Fowler~\cite{no_spi}
gave an example of a fullerene that can not be constructed by this algorithm.
The algorithm described here was the first to prove that the 
counterexample given by Manolopoulos and Fowler~\cite{no_spi} is in fact smallest possible~\cite{min_380}.

The spiral algorithm was later modified to make it complete, but the
resulting algorithm was not efficient~\cite{manolopoulos_92}. In 1995 Yoshida and
Osawa~\cite{yoshida_95} proposed a different algorithm using folding
nets, but its completeness has not been proven.

Other methods are described by Liu et al.~\cite{LKSS91} and Sah~\cite{Sah_93},
but they also didn't lead to sufficiently efficient algorithms.

The most successful approach until now dates from 1997 and is given 
by Brinkmann and Dress~\cite{brinkmann_97}.
The algorithm described there is proven to be complete and has been
implemented in a program called \textit{fullgen}.
The basic strategy can be described as stitching together patches
which are bounded by zigzag (Petrie) paths. Unfortunately a simple
typo-like mistake in the source code produced an error that occurred for the
first time at 136 vertices -- far too many vertices to be detectable 
by any of the
other programs until now. Due to this error the lists in the article of Brinkmann and Dress~\cite{brinkmann_97}
contain some incorrect numbers which we will correct here.

The method of patch replacement can be described as replacing a finite connected region inside some
fullerene with a larger patch with identical boundary. For energetical reasons,
patch replacement as a chemical mechanism to grow fullerenes 
would need very small 
patches. Brinkmann et al.~\cite{brinkmann_06} investigated replacements of small patches and introduced two infinite families of operations. These operations 
can generate all fullerenes up to at
least 200 vertices, but -- as already shown in their paper -- fail in general. 
In 2008
Hasheminezhad, Fleischner and McKay~\cite{hasheminezhad_08} described
a recursive structure using patch replacements for the class of all fullerenes.

In section~\ref{section:generation_nonipr} of this paper we will
describe an algorithm for the efficient generation of all
non-isomorphic fullerenes using the construction operations
from Hasheminezhad et al.~\cite{hasheminezhad_08}. In section~\ref{section:generation_ipr_filter} we will show how to extend this algorithm to generate only IPR fullerenes by using some simple
look-aheads. 


\section{Generation of fullerenes} \label{section:generation_nonipr}
\subsection{The construction algorithm}
We call the patch replacement operations which replace a connected fragment of a fullerene
by a larger fragment \textit{expansions} and the inverse
operations \textit{reductions}. 
If $G'$ is obtained from $G$ by an expansion, we call $G'$ the child of $G$ and $G$ the parent of $G'$.

From the results of Brinkmann et al.~\cite{BGJ06}
it follows that no finite set of patch replacement operations is sufficient to construct 
all fullerenes from smaller ones. So each recursive structure based
on patch replacement operations must necessarily
allow an infinite number of different expansion types.

Hasheminezhad et al.~\cite{hasheminezhad_08} used two infinite families of expansions: $L_i$ and $B_{i,j}$ and a
single expansion $F$. These expansions are sketched in 
Figure~\ref{fig:fullerene_operations}. The lengths of the paths
between the pentagons may vary and for operation $L_i$ 
the mirror image must also be considered. 
All faces drawn completely in
the figure or labelled $f_k$ or $g_k$ have to be distinct. The
faces labelled $f_k$ or $g_k$ can be either pentagons or
hexagons, but when we refer to {\em the} pentagons of the operation, we always
mean the two faces drawn as pentagons. For more details on the expansions
see the article of Hasheminezhad et al.~\cite{hasheminezhad_08}.

\begin{figure}[h!t]
	\centering
	\includegraphics[width=1.0\textwidth]{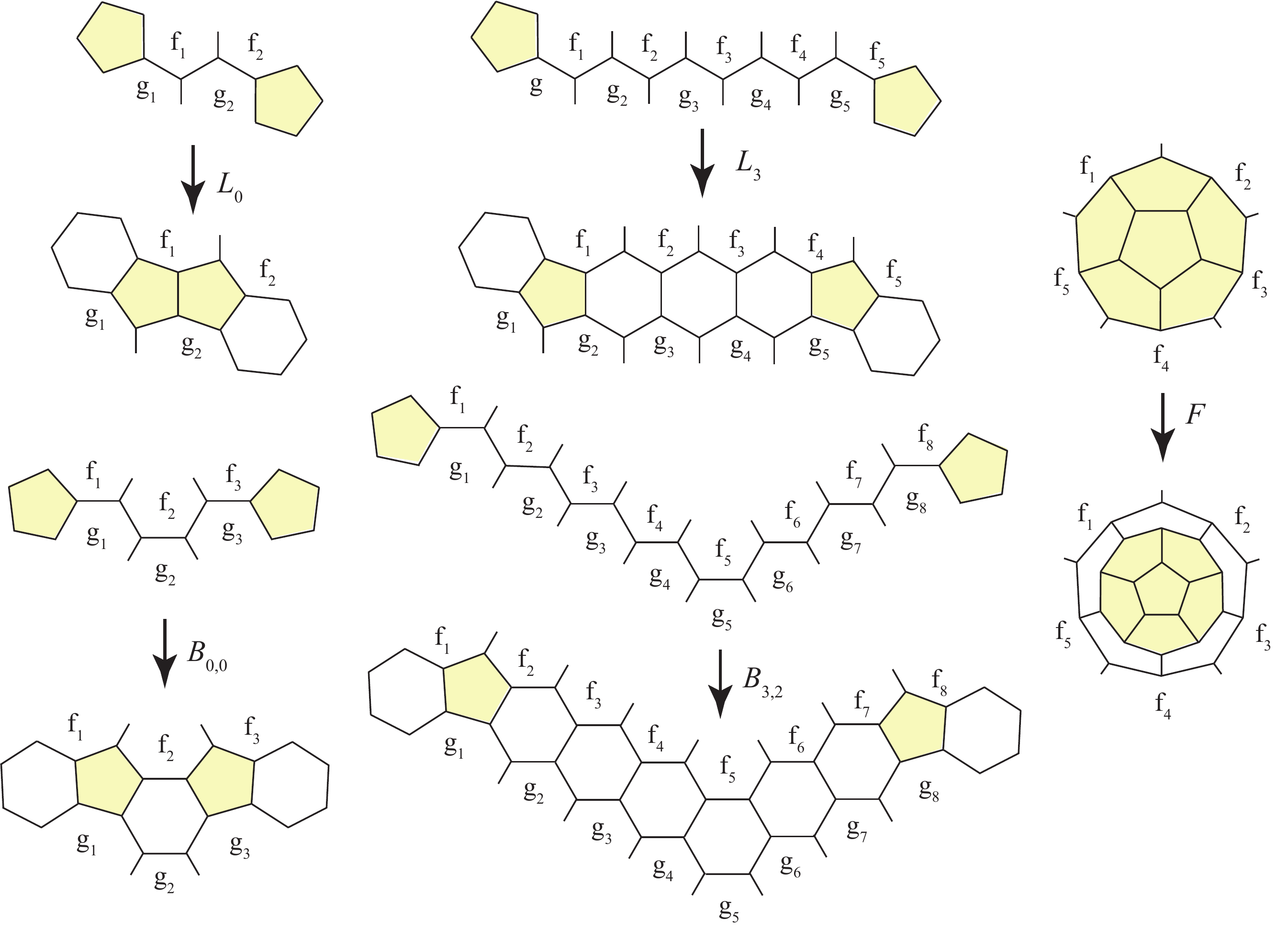}
	\caption{The $L$, $B$ and $F$ expansions for fullerenes.}
	\label{fig:fullerene_operations}
\end{figure}

In Figure~\ref{fig:fullerene_operations-dual} the $L$ and $B$
expansions of
Figure~\ref{fig:fullerene_operations} are shown in dual
representation. We will refer to vertices which have degree $k \in
\{5,6\}$ in the dual representation of a fullerene as $k$-vertices. The
solid white vertices in the figure are 5-vertices, the solid black
vertices are 6-vertices and the shaded vertices can be
either. 


\begin{figure}[h!t]
	\centering
	\resizebox{0.99\textwidth}{!}{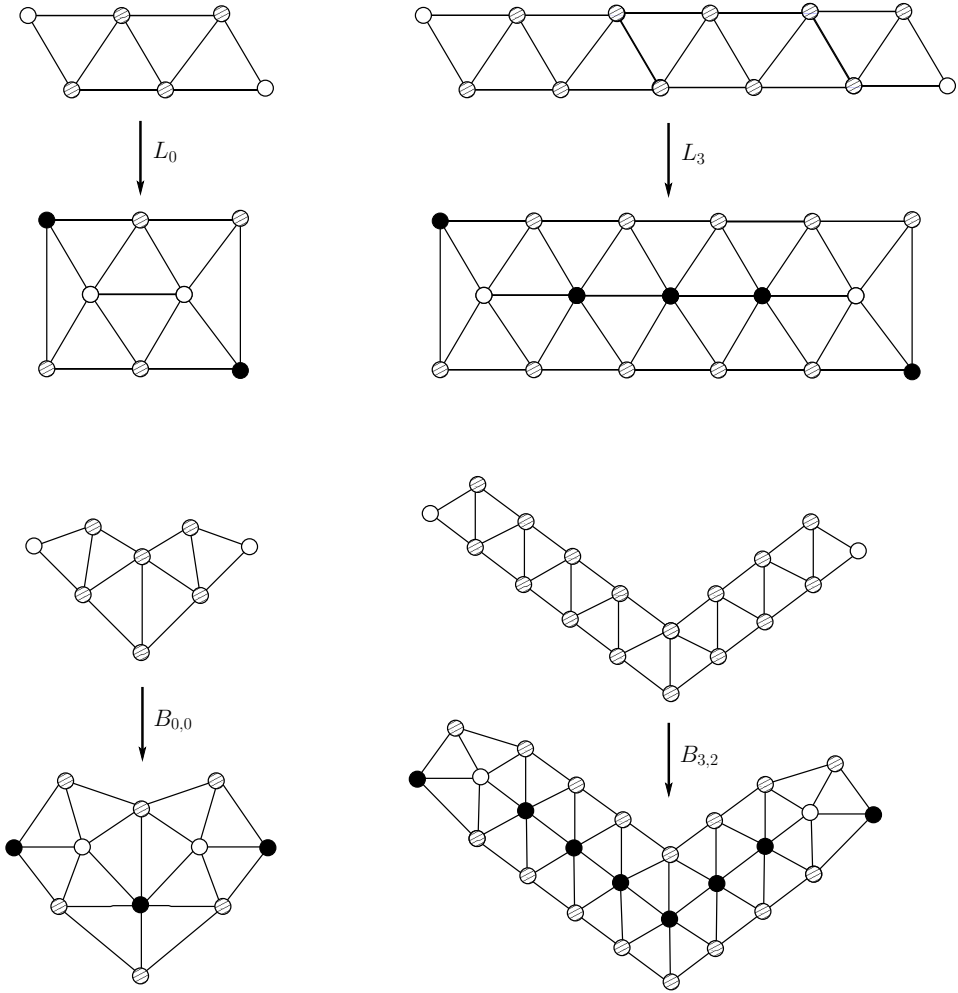}
	\caption{The $L$ and $B$ expansions in dual representation. }
	\label{fig:fullerene_operations-dual}
\end{figure}

Three special fullerenes $C_{20}$ (the dodecahedron),
$C_{28}(T_d)$ and $C_{30}(D_{5h})$ are shown in 
Figure~\ref{fig:irred_fullerenes}.  The type-(5,0) nanotube fullerenes
are those which can be made from $C_{30}(D_{5h})$ by applying
expansion~$F$ zero or more times. 
We will refer to all fullerenes not in one of these classes as
\textit{reducible}.  
The following theorem proved by Hasheminezhad et al.~\cite{hasheminezhad_08} shows
that all reducible fullerenes can be reduced using a type $L$ or $B$
reduction.

\begin{theorem} \label{theorem:construction}
Every fullerene isomer, except $C_{28}(T_d)$ and type-(5,0) nanotube
fullerenes can be constructed by recursively applying expansions of
type $L$ and $B$ to $C_{20}$.
\end{theorem}

Our algorithm uses this theorem by applying $L$ and $B$ expansions
starting at $C_{20}$ and $C_{28}(T_d)$, together with separate (easy) computation
of the type-(5,0) nanotube fullerenes.

%

\begin{figure}[h!t]
	\centering
	\includegraphics[width=0.8\textwidth]{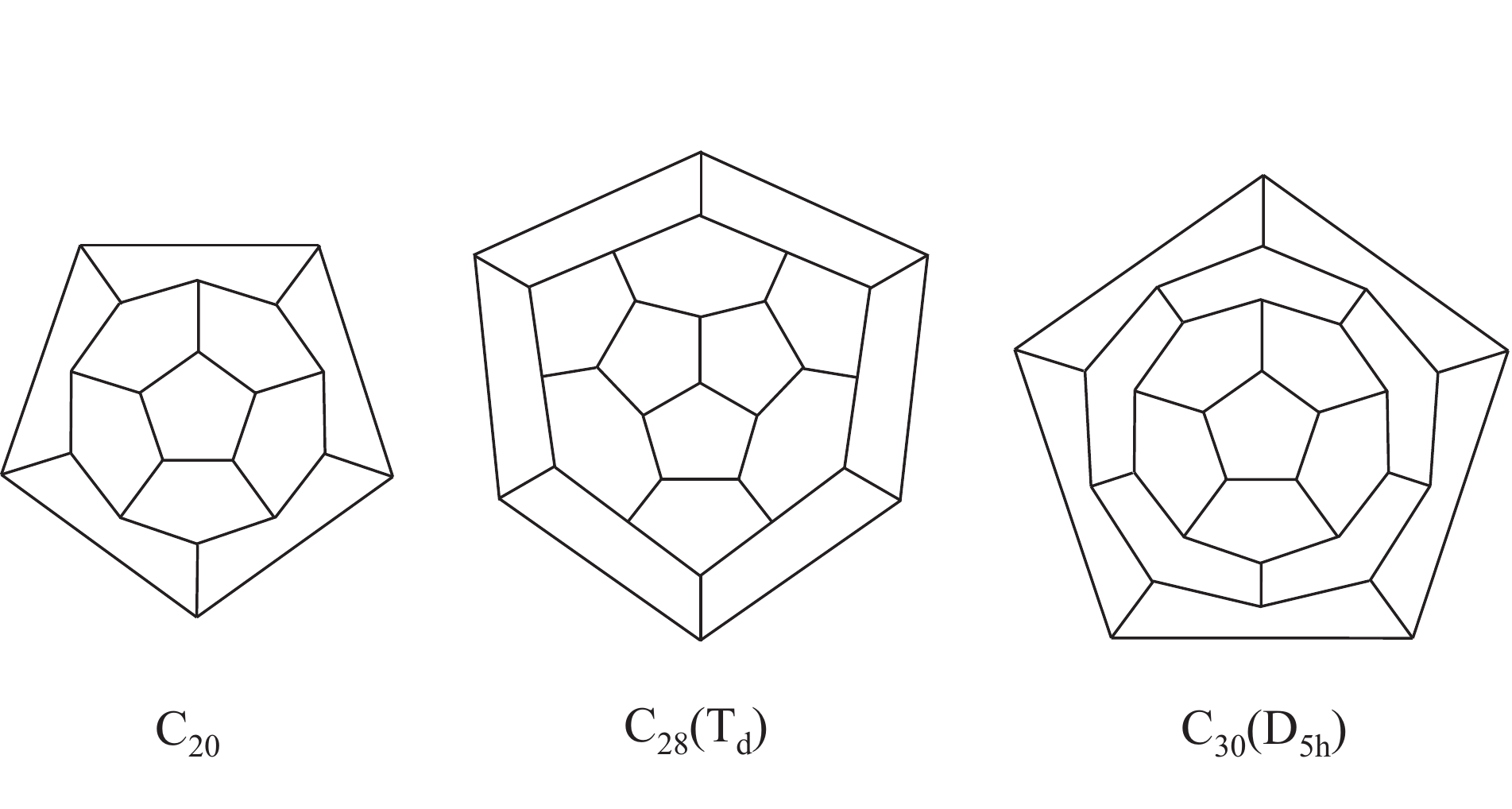}
	\caption{The irreducible fullerenes.}
	\label{fig:irred_fullerenes}
\end{figure}

%

\subsection{Isomorphism rejection and optimizations}

If the expansions are applied in all possible ways,
lots of isomorphic copies will be generated, but we wish
to output only one example of each type. 
We use the canonical construction path method~\cite{mckay_98}, but in the following
we do not assume the reader to be
familiar with the method.

In order to use this method, we first have to define a
\textit{canonical reduction} for every reducible dual fullerene~$G$.
This reduction must be unique up to automorphisms of $G$. 
We call the dual fullerene which is
obtained by applying the canonical reduction to $G$ the
\textit{canonical parent} of $G$ and an expansion that is the inverse
of a canonical reduction in the extended graph a \textit{canonical expansion}.

We also define an equivalence relation on the set of all expansions or reductions
of a given dual fullerene $G$. An expansion is completely characterized by the
patch that is replaced with a larger patch. Two expansions are called
equivalent if there is an automorphism of $G$ mapping the two corresponding
patches onto each other. For reductions, the definition is similar,
but in addition to the patch, a rotational direction
is necessary to uniquely encode a reduction of type $L$. This direction can be
a flag describing whether the new position of the pentagon is in clockwise
or counterclockwise position of the path connecting the pentagons. 
Two type $L$ reductions are equivalent if the patches are mapped onto each other
by an orientation preserving automorphism and the flags are the same or
they are mapped onto each other by an orientation reversing automorphism 
and the flags are different.

The two rules of the canonical construction path method applied to dual fullerenes are:

\begin{enumerate}

\item Only accept a dual fullerene if the last step in its construction was a
canonical expansion. 

\item For each dual fullerene $G$ to which expansions are applied, 
only apply one expansion from each equivalence class of expansions.

\end{enumerate}

The expansions/reductions must of course
be represented in an efficient way.  Reductions are represented
by triple $(e,x,d)$, where $e$ is a directed edge that is the first edge
on the central path between the two pentagons, $x$ is the parameter
set for the reduction (such as ``(2,3)'' for $B_{2,3}$) and $d$ is a
direction.  For $B$ reductions, $d$ indicates whether the turn in
the path is to the left or the right.  For $L$ reductions, $d$ 
distinguishes between this reduction and its mirror image.
Since $e$ can be at either end of the path, there are two equivalent
triples for the same reduction, as illustrated in 
Figure~\ref{fig:representing_triple}.
We call these triples the
\textit{representing triples} of the reduction.
Expansions are also represented by triples in similar fashion.

When we translate the notion of equivalent reductions or
expansions to representing triples,
then the equivalence relation is generated by two relations.
The first is that two triples
are equivalent if they represent the same reduction. The second is that 
$(e,x,d)$ and $(e',x',d')$ are equivalent if $x=x'$ and in case $d=d'$ the edge
$e$ can be mapped to~$e'$ by an orientation preserving automorphism and in case
$d\not =d'$  the edge $e$ can be mapped to~$e'$ by an orientation reversing automorphism.

\begin{figure}[h!t]
	\centering
	\resizebox{0.65\textwidth}{!}{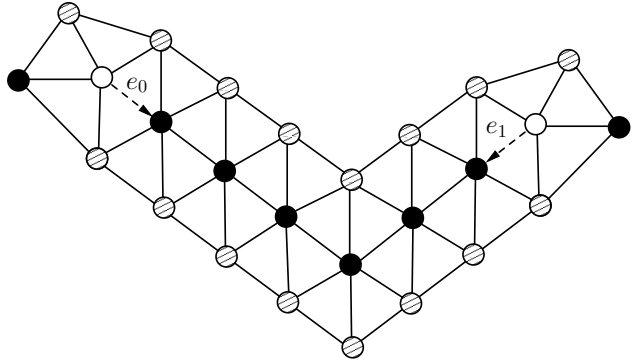}
	\caption{An example of two triples $(e_0, (3,2), 1)$ and $(e_1, (2,3), 0)$ representing the same $B$ reduction.}
	\label{fig:representing_triple}
\end{figure}

For an efficient implementation of the canonicity criteria it is important
that in many cases simple and easily computable criteria can decide on
the canonical reduction or at least reduce the list of possible reductions.
To this end 
we assign a 6-tuple
$(x_0,\dots ,x_5)$ to every triple $(e,p,f)$
representing a possible reduction. We then choose
the canonical reduction to be a reduction which has a representing
triple with the smallest 6-tuple.

The values of $x_0,...,x_4$ are combinatorial invariants of increasing
discriminating power and cost. The value of $x_0$ is the
\textit{length} of the reduction of which $(e,p,f)$ is a representative.
The length of the reduction is the distance between the two
5-vertices of the reduction before actually applying the reduction. 
So in case of a $B_{x,y}$ reduction (2 parameters)
it is $x+y+2$ and in case of an $L_x$ reduction (1 parameter) it is
$x+1$.
Thus we give priority to short reductions. These are easier to detect
and allow some look-ahead. The entry $x_1$ is the
negative of the length of the longest straight path in the reduction.
For an $L$ reduction, the
value of $x_1$ is $-x_0$, which does not distinguish between two
reductions with the same value of $x_0$. For a $B_{x,y}$
reduction it is $- \max\{x,y\}-1$, which sometimes distinguishes 
between $B$ reductions with the same value of $x_0$ and always distinguishes
between an $L$ and a $B$ reduction with the same $x_0$.

The entries $x_2, x_3$ and $x_4$ are strings which contain the degrees
of the vertices in well-defined neighbourhoods of
the edge in the triple. These neighbourhoods are of increasing 
(constant) size.

In each case the value $x_i$ is only computed for those
representing triples
that have the smallest value of $(x_0,...,x_{i-1})$. As our main interest 
is whether an expansion we applied is canonical, we can also stop as soon as we have
found a smaller 6-tuple, which may just mean a reduction with
smaller value of $x_0$.
In case of a unique
triple with minimal value for $(x_0,...,x_{i-1})$ or two such triples
representing the same reduction, we have found the canonical reduction and
can stop the computation of the remaining values. If after the
computation of $(x_0,...,x_{4})$ there is still more than one possibly
canonical triple, we define $x_5$ as a string encoding the whole
structure of the graph relative to the edge and the direction in the representing
triple.  See the article of Brinkmann and McKay~\cite{brinkmann_07} for details of this string, which can be in
short
be described as the code of a BFS-numbering starting at that edge and evaluating
the neighbours of a vertex in the rotational order (clockwise/counterclockwise)
given by the direction. Two triples coding patches in two graphs
(that may be identical or not) containing the same directions are assigned the same value $x_5$
if and only if there is an orientation preserving isomorphism of the graphs mapping the edges
in the triples onto each other. In case of different directions, the same value of $x_5$
is assigned if and only if there is such an orientation reversing automorphism.
This final value $x_5$ makes sure that 
two patches (in the same or different graphs) with the same value of
$(x_0,...,x_{5})$ can be mapped onto each other by an isomorphism $\phi()$
of the graph. When performing the corresponding reductions, the patches 
are replaced by smaller patches and replacing the images $\phi(v)$ 
of vertices inside the patch appropriately, one gets an isomorphism
of the reduced graphs that maps 
the reduced patches onto each other.

When $x_5$ is computed and the graph $G$ that is tested for canonicity 
is accepted,
as a byproduct we also have 
the automorphism group of $G$. 
As possible reductions are represented by edges starting at pentagons, we have a constant
upper bound for the number of possible reductions to be evaluated.
For a given 
triple, each of $x_0,...,x_{4}$ can be computed in constant time and $x_5$
can be computed in linear time, so the canonicity test can be done in linear time.

Even though it is a nice feature that deciding canonicity 
of a given set of possible reductions can be done in linear
time, for the practical performance it is more important that 
computing the combinatorial invariants $(x_0,...,x_4)$ is of
a small constant cost. For dual fullerenes with 152 vertices 
(fullerenes with 300 vertices), the discriminating 
power of $(x_0,...,x_4)$ is enough to
decide whether or not the last expansion was canonical in 
more than 99.9\% of the cases. 

In some cases these cheap invariants also allow look-aheads for deciding
whether or not an expansion can be canonical before actually
performing~it. When making the lists of possible
expansions, we can often already tell that a certain expansion cannot
be canonical since it will not destroy all shorter reductions or since
there will still be a reduction of the same length but with a smaller
value for $x_2$. This avoids the application of a lot of
non-canonical expansions. Counting only expansions passing this
look-ahead, for dual fullerenes with 152 vertices
still in 95.6\% of the expansions a final decision can be found
by only computing $(x_0,...,x_4)$.

If there is only one representing triple with
minimal value for $(x_0,...,x_i)$ $(i \le 4)$, the
automorphism group of $G$ is trivial, so no extra computations are necessary. 
This happens in 80.9\% of the cases for dual fullerenes with 152 vertices.
The ratio is decreasing with the number of vertices.
For 102 vertices of the dual fullerene it is 93.3\% and for 127 vertices 
it is 86.9\%.

\begin{theorem}
Assume that exactly one representative of each isomorphism class of 
dual fullerenes on up to $n-2$ vertices is given.
Suppose we perform the following steps:
\begin{enumerate}\itemsep=0pt
\item Perform one expansion of each
equivalence class of $L$ and $B$ expansions which lead to a 
dual fullerene with $n$ vertices.

\item Accept each new dual fullerene if and only if a triple
representing the inverse of the last expansion has the minimal value of
$(x_0,...,x_5)$ among all possible reductions.
\end{enumerate}

Then exactly one representative of each isomorphism
class of reducible dual fullerenes with $n$ vertices is accepted.
\end{theorem}

\begin{proof}

Let $G$ be a reducible dual fullerene with $n$ vertices.
By Theorem~\ref{theorem:construction} there is at least one
reduction, and so a canonical reduction~$\rho$, that applies to $G$.
The graph resulting from $\rho$ is isomorphic to a graph in the input set,
which has an expansion which is equivalent to the inverse of $\rho$. 
But this expansion produces
a graph isomorphic to $G$ and the parameters of its inverse reduction
are the same as those of $\rho$, so the result of
the expansion is accepted.

This implies that at least one representative of each isomorphism
class in question is generated. It remains to be shown that at most
one is generated.

Suppose that the algorithm accepts two isomorphic
fullerenes $G$ and $G'$ with $n$ vertices. As they are isomorphic, the
canonical reductions have the same parameter set $(x_0,...,x_{5})$.
As they were both accepted, they were constructed by a
canonical expansion, so 
 -- as mentioned before -- the two parents $G_0$ and $G'_0$ 
are isomorphic and there is an
isomorphism that maps the corresponding expansions onto each other. 
By our assumption this means that
$G_0$ and $G'_0$ are identical and that the two expansions are
equivalent, which contradicts step~1.
\end{proof}

By recursively applying expansion $F$ to $C_{20}$, all
type-(5,0) fullerenes are constructed. As this constructs all
type-(5,0) fullerenes exactly once and these fullerenes can not
be constructed by $L$ or $B$ expansions, this completes the algorithm.

\subsection{Optimizations}

As most fullerenes contain short reductions and as
we give priority to short reductions, by far most long expansions
are not canonical. For efficiency reasons it is interesting
to determine an upper bound on the length of a canonical expansion.

\begin{lemma} \label{lemma_reduction_nonipr}
Reducible dual fullerenes which contain adjacent 5-vertices have an $L_0$, $L_1$ or $B_{0,0}$ reduction.
\end{lemma}

\begin{proof}
For a proof, see the article of Hasheminezhad et al.~\cite{hasheminezhad_08}.
\end{proof}

So each reducible dual non-IPR fullerene has a reduction with length at most~2.
In dual IPR fullerenes the shortest reduction is
a reduction with the same length as the minimum distance of two 5-vertices
in the dual fullerene.

In dual fullerenes where the
shortest distance between two 5-vertices is at least $d$,
the sets of vertices at distance at most $\lfloor \frac{d-1}{2} \rfloor$
of different vertices are disjoint.
This gives us a lower bound of $12\, f(\lfloor \frac{d-1}{2} \rfloor)$ for
the number of vertices in the fullerene, where $f(x) = 1 + \frac52 (x+1)x$.

So expansions of length $d$ are not canonical if the expanded
graph contains fewer than $12\, f(\lfloor \frac{d-1}{2} \rfloor)$ vertices.
This result does not only help to avoid the application of non-canonical 
expansions, but also avoids the need to search for long expansions.

We can often determine even sharper upper
bounds for the maximum length of a canonical expansion:

\begin{lemma} \label{lemma_one_L1}
If a dual fullerene $G$ has a reduction of length $d\le 2$, all children $G'$
of $G$ have a reduction of length at most $d+2$.
\end{lemma}

\begin{proof}
If $G'$ is not IPR, this follows from
Lemma~\ref{lemma_reduction_nonipr}, so assume that $G'$ is IPR.
The length of the shortest reduction is then the shortest distance between two
5-vertices. Let us look at the shortest path $W$ between two 5-vertices
allowing a reduction of length $d$ in $G$. 

As $d\le 2$ and as all vertices in the patch $P$ used for expansion must be distinct,
$W$ can contain at most 2 maximal subpaths entering $P$ and ending there, 
starting in $P$ and leaving it or crossing $P$.

The distance between a 5-vertex in $P$ from vertices on the boundary
grows at most by 1.  The same is true for each pair of vertices on the
boundary. So the path $W$ can grow in two places by at most $1$,
proving the result. 
\end{proof}

This lemma could be proven for larger $d$ if one required the child to be
canonical, but as $12 \, f(\frac{5-1}{2}) = 192$, all dual fullerenes
with less than 192 vertices (or fullerenes with 380 vertices) have a
reduction of length at most~4. Therefore, even for $d=2$,
Lemma~\ref{lemma_one_L1} is
only useful for fullerenes with at least 380 vertices. 

\begin{lemma} \label{lemma_one_L0}
If a dual fullerene $G$ has an $L_0$ reduction, all canonical 
children $G'$ of $G$ have a reduction of length at most 2.
\end{lemma}

\begin{proof}
If $G'$ is not IPR, this follows from
Lemma~\ref{lemma_reduction_nonipr}, so assume that $G'$ is IPR. 
By Lemma~\ref{lemma_one_L1}, $G'$ has a reduction of length
at most 3, so a canonical child was constructed by an expansion
of length at most 3. If $G'$ was constructed by an $L_0$, $L_1$ or $B_{0,0}$ expansion, the statement follows immediately.

Figure~\ref{fig:L0_red_expansion_L2}
and Figure~\ref{fig:L0_red_expansion_B10} show the only ways that an $L_2$
(resp.\ $B_{1,0}$) expansion can destroy an $L_0$ reduction which
involves two pentagons $p_1$ and $p_2$ such that the expanded
fullerene $G'$ contains no reduction of length shorter than 3. 
The
faces $f_i$ and $g_i$ $(1 \le i \le 4)$ which are on the boundary of
the $L_2$ or $B_{1,0}$ expansion have to be hexagons otherwise the
dual of $G'$ would contain 5-vertices which are at distance at most
2. Since $p_1$ and $p_2$ are involved in the $L_0$ reduction, they
must share an edge. So there is an edge $a \in \{e_1, e_2, e_3\}$
which is equal to an edge $b \in \{e_4, e_5, e_6\}$ and as the pentagons share an
edge, they must also share two faces each 
containing an endpoint of this common edge. It is easy to see
that for all possible choices of $a$ and $b$ this implies that a
fullerene containing a patch from
Figure~\ref{fig:L0_red_expansion_L2} or
Figure~\ref{fig:L0_red_expansion_B10} must have a
4-edge-cut or a 5-edge-cut. However it follows from the results of Bornh\"oft et al.~\cite{bornhoft_03} that fullerenes are cyclically 5-edge connected,
so 4-edge-cuts do not exist. Kardo\u{s} and \u{S}krekovski~\cite{kardos_08} showed that the
type-(5,0) nanotubes are the only fullerenes which have a non-trivial
5-edge-cut.

So there is no expansion which can be applied to $G$ such that the
shortest reduction of the expanded fullerene has length 3. Thus all
canonical children of $G$ have a reduction of length at most 2.
\end{proof}

\begin{figure}[h!t]
	\centering
	\resizebox{0.5\textwidth}{!}{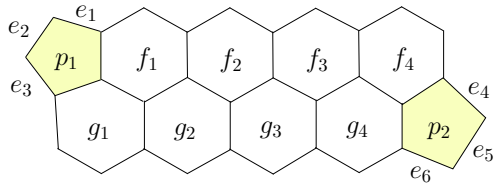}
	\caption{The initial patch of an $L_2$ expansion involving two neighbouring pentagons $p_1$ and $p_2$. One of the edges from $\{e_1, e_2, e_3\}$ is equal to an edge in $\{e_4, e_5, e_6\}$.}
	\label{fig:L0_red_expansion_L2}
\end{figure}

\begin{figure}[h!t]
	\centering
	\resizebox{0.5\textwidth}{!}{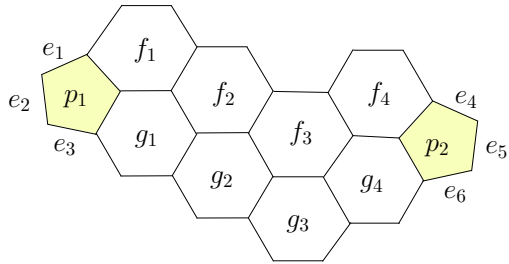}
	\caption{The initial patch of a $B_{1,0}$ expansion involving two neighbouring pentagons $p_1$ and $p_2$. One of the edges from $\{e_1, e_2, e_3\}$ is equal to an edge in $\{e_4, e_5, e_6\}$.}
	\label{fig:L0_red_expansion_B10}
\end{figure}

For the next lemmas the following observation is useful:

\begin{obs}\label{obs:3inpatch}

If the set of vertices contained in the initial patch of an expansion of length $l$ contains
at least three 5-vertices (so in addition to the two 5-vertices of the expansion
there is at least one
more 5-vertex in the boundary), then in the extended patch there are two 5-vertices at distance at most $l/2 +1$.

\end{obs}

\begin{lemma} \label{lemma_two_L1s}
If a dual fullerene $G$ has at least two reductions of length 2 which do not have
the same set of 5-vertices of the reduction,
all canonical children $G'$ have a reduction of length at most~3.
\end{lemma}

\begin{proof}
If $G'$ is not a dual IPR fullerene, the result follows immediately, so assume the opposite. This implies that we 
have to find a bound for the minimum distance of two 5-vertices.
By Lemma~\ref{lemma_one_L1} each child has a reduction of length $4$. So each canonical child was constructed by an 
expansion of length at most~$4$. If there were three 5-vertices in the initial patch of the expansion, the result
follows from Observation~\ref{obs:3inpatch}. So assume this is not the case and one 5-vertex of a
reduction of length 2 is not contained in the initial patch. But then the distance to the other 5-vertex
in the reduction can grow by at most~1, proving the lemma.
\end{proof}

\begin{lemma} \label{lemma_three_L1s}
If a dual fullerene $G$ has at least three reductions of length 2 with pairwise disjoint sets
of 5-vertices of the reduction, all canonical children $G'$ of $G$ have a reduction of length at most~2.
\end{lemma}

\begin{proof}
We may again assume that $G'$ is IPR. It follows from Lemma~\ref{lemma_two_L1s} that $G'$ has a reduction of length at most 3,
so each canonical expansion has length at most $3$. If there are three 5-vertices in the initial patch of the expansion, the result
follows directly from the observation. So there is (at least) one reduction of length 2 so that none of its 5-vertices
is contained in that initial patch. But then the path of length 2 between these 5-vertices still exists in the expanded graph
and allows a reduction of length $2$.
\end{proof}

For two reductions $R_1$ and $R_2$ in a dual fullerene $G$ we define the distance
$d(R_1, R_2)$ to be $\min\{d(a_1,a_2) \mid a_i \mbox{ is a 5-vertex of } R_i\}$.

\begin{lemma} \label{lemma_two_ind_L0s}
If a dual fullerene $G$ has $L_0$ reductions $R_1$ and $R_2$ with $d(R_1, R_2)>4$, all canonical children
$G'$ of $G$ have an $L_0$ reduction.
\end{lemma}

\begin{proof}
It follows from Lemma~\ref{lemma_one_L0} that there is a reduction of
length at most~2 in $G'$. The distance between vertices which are in the initial patch
of an expansion of length 2 is at most
4. Therefore at least one of the two neighbouring 5-vertex pairs still exists
and the neighbouring vertices are either unchanged or changed to 6-vertices. In either case
the reduction will still be possible.
\end{proof}

For dual fullerenes with $152$ vertices, Lemmas~\ref{lemma_one_L1},
\ref{lemma_two_L1s}, \ref{lemma_three_L1s} and \ref{lemma_two_ind_L0s}
can be used to determine a bound on the length of canonical expansions
in 93.9\% of the cases.


\section{Generation of IPR fullerenes} \label{section:generation_ipr_filter}

The algorithm was developed for generating all fullerenes, but it can
also be used to generate only IPR fullerenes
by using a filter and some simple look-aheads:

An $L_0$ expansion is the only expansion that increases the number
of vertices in a dual fullerene by just $2$ vertices, but the result 
of an $L_0$ expansion is never a dual IPR-Fullerene. When constructing
dual IPR fullerenes with $n$ vertices, dual IPR fullerenes with $n-2$
vertices do not have to be constructed and the largest dual fullerenes 
to which an expansion is applied have $n-3$ vertices.

For a dual fullerene with $n-4$ vertices only expansions of length 3
(i.e. $L_2$ or $B_{1,0}$ expansions) can lead to dual IPR fullerenes
with $n$ vertices. However if a dual fullerene with $n-4$ vertices
contains an $L_0$ reduction, it follows from Lemma~\ref{lemma_one_L0}
that expansions
of length 3 are not canonical.  Thus we can reject
all dual fullerenes with $n-4$
vertices that contain an $L_0$ reduction and also avoid applying
$L_0$ expansions to dual fullerenes with $n-6$ vertices.

Already these simple look-aheads result in an efficient program, as can be seen
in Table~\ref{table:fullerene_times}.


\section{Testing and results}

The running times and a comparison
with \textit{fullgen} are given in Table~\ref{table:fullerene_times}. Our
generator is called \textit{buckygen}. The program was compiled with
gcc and executed in a single thread
on an Intel Xeon L5520 CPU at 2.27 GHz. The running
times include writing the fullerenes to a null device.

\textit{Buckgen} was used to generate all fullerenes up to 400
vertices.  This led to a programming error being uncovered in
\textit{fullgen} that caused it to miss some fullerenes starting
at 136 vertices and IPR fullerenes starting at 254 vertices.
After correction of the error in \textit{fullgen}, the two programs agree
to at least 380 vertices, which is a good check of both.
We give the counts in 
Tables~\ref{table:degree_counts_1}--\ref{table:degree_counts_5},
which correct those in the article of Brinkmann and Dress~\cite{brinkmann_97} where they overlap.
The fullerenes themselves can be downloaded from 
\url{http://hog.grinvin.org/Fullerenes} for small sizes.

We also repeated and extended a computation reported by Brinkmann et al.~\cite{brinkmann_03}, 
which relied on the faulty version of \textit{fullgen}, the results are listed in Tables~\ref{table:degree_counts_1}--\ref{table:degree_counts_4}.  Now we have confirmed
that all cubic planar graphs with maximum face size 6 are hamiltonian to
at least 316 vertices, in agreement with the famous conjecture of Barnette.

The incomplete lists of fullerenes were also used in another article of Brinkmann et al.~\cite{brinkmann_06}. All reducibility results given there remain true, except for Table~2, where the number of fullerenes that can not be reduced by a growth operation of cost 7 -- that is replacing only 7 edges -- is 1 too small for 186 and 190 vertices and 2 too small for 194 vertices.

\begin{table}
\begin{center}
\begin{tabular}{| c || c | c | c | c |}
\hline 
number of & fullerenes/s & fullgen (s) / & IPR fullerenes/s & fullgen IPR (s) /\\
vertices & (buckygen) & buckygen (s) & (buckygen) & buckygen IPR (s)\\
\hline
100  &  42 358  &  7.30  &  105  &  0.28\\
140  &  33 369  &  7.39  &  789  &  0.46\\
170  &  21 268  &  5.63  &  1 174  &  0.58\\
200  &  16 953  &  5.49  &  1 630  &  0.80\\
230  &  12 597  &  5.13  &  1 721  &  0.96\\
260  &  9 408  &  4.59  &  1 632  &  1.03\\
280  &  7 735  &  4.43  &  1 530  &  1.10\\
300  &  6 494  &  4.07  &  1 425  &  1.16\\
320  &  5 502  &  3.67  &  1 332  &  1.14\\
\hline
20--100  &  159 365  &  24.96  &  278  &  0.77\\
102--150  &  157 736  &  33.04  &  4 643  &  2.44\\
152--200  &  115 625  &  32.08  &  10 558  &  4.71\\
202--250  &  82 813  &  32.09  &  13 212  &  6.84\\
\hline
\end{tabular}
\end{center}

\caption{Generation rates for fullerenes.}

\label{table:fullerene_times}
\end{table}

\begin{table}
{\small 
\begin{center}
\begin{tabular}{| c | c | c | c | c | c |}
\hline 
 nv & nf & min.\,face 3 & min.\,face 4 & fullerenes & IPR fullerenes\\
\hline 
4  &  4  &  1  &  0  &  0  &  0\\
6  &  5  &  1  &  0  &  0  &  0\\
8  &  6  &  1  &  1  &  0  &  0\\
10  &  7  &  4  &  1  &  0  &  0\\
12  &  8  &  8  &  2  &  0  &  0\\
14  &  9  &  11  &  4  &  0  &  0\\
16  &  10  &  23  &  7  &  0  &  0\\
18  &  11  &  34  &  10  &  0  &  0\\
20  &  12  &  54  &  22  &  1  &  0\\
22  &  13  &  83  &  32  &  0  &  0\\
24  &  14  &  125  &  58  &  1  &  0\\
26  &  15  &  174  &  92  &  1  &  0\\
28  &  16  &  267  &  151  &  2  &  0\\
30  &  17  &  365  &  227  &  3  &  0\\
32  &  18  &  509  &  368  &  6  &  0\\
34  &  19  &  706  &  530  &  6  &  0\\
36  &  20  &  963  &  805  &  15  &  0\\
38  &  21  &  1 270  &  1 158  &  17  &  0\\
40  &  22  &  1 708  &  1 695  &  40  &  0\\
42  &  23  &  2 204  &  2 373  &  45  &  0\\
44  &  24  &  2 876  &  3 354  &  89  &  0\\
46  &  25  &  3 695  &  4 595  &  116  &  0\\
48  &  26  &  4 708  &  6 340  &  199  &  0\\
50  &  27  &  5 925  &  8 480  &  271  &  0\\
52  &  28  &  7 491  &  11 417  &  437  &  0\\
54  &  29  &  9 255  &  15 049  &  580  &  0\\
56  &  30  &  11 463  &  19 832  &  924  &  0\\
58  &  31  &  14 083  &  25 719  &  1 205  &  0\\
60  &  32  &  17 223  &  33 258  &  1 812  &  1\\
62  &  33  &  20 857  &  42 482  &  2 385  &  0\\
64  &  34  &  25 304  &  54 184  &  3 465  &  0\\
66  &  35  &  30 273  &  68 271  &  4 478  &  0\\
68  &  36  &  36 347  &  85 664  &  6 332  &  0\\
70  &  37  &  43 225  &  106 817  &  8 149  &  1\\
72  &  38  &  51 229  &  132 535  &  11 190  &  1\\
74  &  39  &  60 426  &  163 194  &  14 246  &  1\\
76  &  40  &  71 326  &  200 251  &  19 151  &  2\\
78  &  41  &  83 182  &  244 387  &  24 109  &  5\\
80  &  42  &  97 426  &  296 648  &  31 924  &  7\\
82  &  43  &  113 239  &  358 860  &  39 718  &  9\\
84  &  44  &  131 425  &  431 578  &  51 592  &  24\\
86  &  45  &  151 826  &  517 533  &  63 761  &  19\\
88  &  46  &  175 302  &  617 832  &  81 738  &  35\\
\hline
\end{tabular}
\end{center}
}
\caption{Cubic plane graphs with maximum face size $6$ listed with respect to their minimum face size. Cubic plane graphs with maximum face size $6$ and with minimum face size 5 are fullerenes. nv is the number of vertices and nf is the number of faces.}

\label{table:degree_counts_1}

\end{table}

\begin{table}
{\small 
\begin{center}
\begin{tabular}{| c | c | c | c | c | c |}
\hline 
 nv & nf & min.\,face 3 & min.\,face 4 & fullerenes & IPR fullerenes\\
\hline 
90  &  47  &  200 829  &  735 257  &  99 918  &  46\\
92  &  48  &  231 042  &  870 060  &  126 409  &  86\\
94  &  49  &  263 553  &  1 029 114  &  153 493  &  134\\
96  &  50  &  300 602  &  1 209 783  &  191 839  &  187\\
98  &  51  &  341 960  &  1 420 472  &  231 017  &  259\\
100  &  52  &  388 673  &  1 659 473  &  285 914  &  450\\
102  &  53  &  438 795  &  1 937 509  &  341 658  &  616\\
104  &  54  &  496 961  &  2 249 285  &  419 013  &  823\\
106  &  55  &  559 348  &  2 612 410  &  497 529  &  1 233\\
108  &  56  &  629 807  &  3 015 386  &  604 217  &  1 799\\
110  &  57  &  706 930  &  3 483 289  &  713 319  &  2 355\\
112  &  58  &  792 703  &  4 002 504  &  860 161  &  3 342\\
114  &  59  &  885 137  &  4 600 343  &  1 008 444  &  4 468\\
116  &  60  &  990 929  &  5 257 856  &  1 207 119  &  6 063\\
118  &  61  &  1 102 609  &  6 019 580  &  1 408 553  &  8 148\\
120  &  62  &  1 227 043  &  6 849 385  &  1 674 171  &  10 774\\
122  &  63  &  1 363 825  &  7 805 813  &  1 942 929  &  13 977\\
124  &  64  &  1 513 612  &  8 846 570  &  2 295 721  &  18 769\\
126  &  65  &  1 673 568  &  10 041 875  &  2 650 866  &  23 589\\
128  &  66  &  1 853 928  &  11 335 288  &  3 114 236  &  30 683\\
130  &  67  &  2 045 154  &  12 821 597  &  3 580 637  &  39 393\\
132  &  68  &  2 255 972  &  14 415 241  &  4 182 071  &  49 878\\
134  &  69  &  2 485 363  &  16 248 586  &  4 787 715  &  62 372\\
136  &  70  &  2 732 106  &  18 211 371  &  5 566 949  &  79 362\\
138  &  71  &  2 998 850  &  20 454 114  &  6 344 698  &  98 541\\
140  &  72  &  3 295 090  &  22 845 387  &  7 341 204  &  121 354\\
142  &  73  &  3 606 102  &  25 587 469  &  8 339 033  &  151 201\\
144  &  74  &  3 944 923  &  28 486 985  &  9 604 411  &  186 611\\
146  &  75  &  4 316 999  &  31 808 776  &  10 867 631  &  225 245\\
148  &  76  &  4 711 038  &  35 313 026  &  12 469 092  &  277 930\\
150  &  77  &  5 135 794  &  39 315 258  &  14 059 174  &  335 569\\
152  &  78  &  5 599 065  &  43 529 295  &  16 066 025  &  404 667\\
154  &  79  &  6 091 434  &  48 339 505  &  18 060 979  &  489 646\\
156  &  80  &  6 621 013  &  53 361 979  &  20 558 767  &  586 264\\
158  &  81  &  7 198 926  &  59 117 693  &  23 037 594  &  697 720\\
160  &  82  &  7 800 960  &  65 110 208  &  26 142 839  &  836 497\\
162  &  83  &  8 460 776  &  71 938 170  &  29 202 543  &  989 495\\
164  &  84  &  9 168 333  &  79 041 733  &  33 022 573  &  1 170 157\\
166  &  85  &  9 917 772  &  87 147 815  &  36 798 433  &  1 382 953\\
168  &  86  &  10 711 603  &  95 517 631  &  41 478 344  &  1 628 029\\
170  &  87  &  11 590 680  &  105 090 752  &  46 088 157  &  1 902 265\\
172  &  88  &  12 491 734  &  114 936 807  &  51 809 031  &  2 234 133\\
174  &  89  &  13 479 003  &  126 169 808  &  57 417 264  &  2 601 868\\
176  &  90  &  14 518 882  &  137 732 548  &  64 353 269  &  3 024 383\\
\hline
\end{tabular}
\end{center}
}
\caption{Cubic plane graphs with maximum face size $6$ listed with respect to their minimum face size (continued). nv is the number of vertices and nf is the number of faces.}

\label{table:degree_counts_2}

\end{table}

\begin{table}
{\small 
\begin{center}
\begin{tabular}{| c | c | c | c | c |}
\hline 
 nv & nf & min.\,face 4 & fullerenes & IPR fullerenes\\
\hline 
178  &  91  &  150 895 768  &  71 163 452  &  3 516 365\\
180  &  92  &  164 343 840  &  79 538 751  &  4 071 832\\
182  &  93  &  179 752 024  &  87 738 311  &  4 690 880\\
184  &  94  &  195 420 760  &  97 841 183  &  5 424 777\\
186  &  95  &  213 287 269  &  107 679 717  &  6 229 550\\
188  &  96  &  231 489 614  &  119 761 075  &  7 144 091\\
190  &  97  &  252 233 869  &  131 561 744  &  8 187 581\\
192  &  98  &  273 226 069  &  145 976 674  &  9 364 975\\
194  &  99  &  297 264 792  &  159 999 462  &  10 659 863\\
196  &  100  &  321 450 554  &  177 175 687  &  12 163 298\\
198  &  101  &  349 098 672  &  193 814 658  &  13 809 901\\
200  &  102  &  376 999 869  &  214 127 742  &  15 655 672\\
202  &  103  &  408 774 872  &  233 846 463  &  17 749 388\\
204  &  104  &  440 627 726  &  257 815 889  &  20 070 486\\
206  &  105  &  477 200 827  &  281 006 325  &  22 606 939\\
208  &  106  &  513 632 380  &  309 273 526  &  25 536 557\\
210  &  107  &  555 304 108  &  336 500 830  &  28 700 677\\
212  &  108  &  596 974 072  &  369 580 714  &  32 230 861\\
214  &  109  &  644 526 803  &  401 535 955  &  36 173 081\\
216  &  110  &  691 786 828  &  440 216 206  &  40 536 922\\
218  &  111  &  746 085 995  &  477 420 176  &  45 278 722\\
220  &  112  &  799 648 739  &  522 599 564  &  50 651 799\\
222  &  113  &  861 133 064  &  565 900 181  &  56 463 948\\
224  &  114  &  922 082 216  &  618 309 598  &  62 887 775\\
226  &  115  &  991 650 902  &  668 662 698  &  69 995 887\\
228  &  116  &  1 060 208 550  &  729 414 880  &  77 831 323\\
230  &  117  &  1 139 239 947  &  787 556 069  &  86 238 206\\
232  &  118  &  1 216 496 915  &  857 934 016  &  95 758 929\\
234  &  119  &  1 305 306 936  &  925 042 498  &  105 965 373\\
236  &  120  &  1 392 596 607  &  1 006 016 526  &  117 166 528\\
238  &  121  &  1 492 525 091  &  1 083 451 816  &  129 476 607\\
240  &  122  &  1 590 214 959  &  1 176 632 247  &  142 960 479\\
242  &  123  &  1 702 998 124  &  1 265 323 971  &  157 402 781\\
244  &  124  &  1 812 247 954  &  1 372 440 782  &  173 577 766\\
246  &  125  &  1 938 356 975  &  1 474 111 053  &  190 809 628\\
\hline
\end{tabular}
\end{center}
}
\caption{Triangle-free cubic plane graphs with maximum face size $6$ listed with
  respect to their minimum face size. nv is the number of vertices and nf is the number of faces.}

\label{table:degree_counts_3}

\end{table}

\begin{table}
{\small 
\begin{center}
\begin{tabular}{| c | c | c | c | c |}
\hline 
 nv & nf & min.\,face 4 & fullerenes & IPR fullerenes\\
\hline 
248  &  126  &  2 061 311 003  &  1 596 482 232  &  209 715 141\\
250  &  127  &  2 202 202 308  &  1 712 934 069  &  230 272 559\\
252  &  128  &  2 338 869 735  &  1 852 762 875  &  252 745 513\\
254  &  129  &  2 497 257 527  &  1 985 250 572  &  276 599 787\\
256  &  130  &  2 649 382 974  &  2 144 943 655  &  303 235 792\\
258  &  131  &  2 825 361 014  &  2 295 793 276  &  331 516 984\\
260  &  132  &  2 995 557 818  &  2 477 017 558  &  362 302 637\\
262  &  133  &  3 191 292 821  &  2 648 697 036  &  395 600 325\\
264  &  134  &  3 379 722 482  &  2 854 536 850  &  431 894 257\\
266  &  135  &  3 598 542 661  &  3 048 609 900  &  470 256 444\\
268  &  136  &  3 806 922 124  &  3 282 202 941  &  512 858 451\\
270  &  137  &  4 049 087 424  &  3 501 931 260  &  557 745 670\\
272  &  138  &  4 281 540 754  &  3 765 465 341  &  606 668 511\\
274  &  139  &  4 549 259 510  &  4 014 007 928  &  659 140 287\\
276  &  140  &  4 805 073 991  &  4 311 652 376  &  716 217 922\\
278  &  141  &  5 103 457 703  &  4 591 045 471  &  776 165 188\\
280  &  142  &  5 385 296 261  &  4 926 987 377  &  842 498 881\\
282  &  143  &  5 713 728 893  &  5 241 548 270  &  912 274 540\\
284  &  144  &  6 026 548 238  &  5 618 445 787  &  987 874 095\\
286  &  145  &  6 388 285 729  &  5 972 426 835  &  1 068 507 788\\
288  &  146  &  6 731 485 975  &  6 395 981 131  &  1 156 161 307\\
290  &  147  &  7 132 734 985  &  6 791 769 082  &  1 247 686 189\\
292  &  148  &  7 508 699 038  &  7 267 283 603  &  1 348 832 364\\
294  &  149  &  7 948 994 131  &  7 710 782 991  &  1 454 359 806\\
296  &  150  &  8 365 304 423  &  8 241 719 706  &  1 568 768 524\\
298  &  151  &  8 847 679 520  &  8 738 236 515  &  1 690 214 836\\
300  &  152  &  9 302 042 370  &  9 332 065 811  &  1 821 766 896\\
302  &  153  &  9 835 862 103  &  9 884 604 767  &  1 958 581 588\\
304  &  154  &  10 332 102 625  &  10 548 218 751  &  2 109 271 290\\
306  &  155  &  10 915 020 041  &  11 164 542 762  &  2 266 138 871\\
308  &  156  &  11 462 133 758  &  11 902 015 724  &  2 435 848 971\\
310  &  157  &  12 098 825 145  &  12 588 998 862  &  2 614 544 391\\
312  &  158  &  12 694 519 224  &  13 410 330 482  &  2 808 510 141\\
314  &  159  &  13 396 207 247  &  14 171 344 797  &  3 009 120 113\\
316  &  160  &  14 043 402 497  &  15 085 164 571  &  3 229 731 630\\
\hline
\end{tabular}
\end{center}
}
\caption{Triangle-free cubic plane graphs with maximum face size $6$ listed with
  respect to their minimum face size (continued). nv is the number of vertices and nf is the number of faces.}

\label{table:degree_counts_4}

\end{table}

\begin{table}
{\small 
\begin{center}
\begin{tabular}{| c | c | c | c |}
\hline 
 nv & nf & fullerenes & IPR fullerenes\\
\hline 
318  &  161  &  15 930 619 304  &  3 458 148 016\\
320  &  162  &  16 942 010 457  &  3 704 939 275\\
322  &  163  &  17 880 232 383  &  3 964 153 268\\
324  &  164  &  19 002 055 537  &  4 244 706 701\\
326  &  165  &  20 037 346 408  &  4 533 465 777\\
328  &  166  &  21 280 571 390  &  4 850 870 260\\
330  &  167  &  22 426 253 115  &  5 178 120 469\\
332  &  168  &  23 796 620 378  &  5 531 727 283\\
334  &  169  &  25 063 227 406  &  5 900 369 830\\
336  &  170  &  26 577 912 084  &  6 299 880 577\\
338  &  171  &  27 970 034 826  &  6 709 574 675\\
340  &  172  &  29 642 262 229  &  7 158 963 073\\
342  &  173  &  31 177 474 996  &  7 620 446 934\\
344  &  174  &  33 014 225 318  &  8 118 481 242\\
346  &  175  &  34 705 254 287  &  8 636 262 789\\
348  &  176  &  36 728 266 430  &  9 196 920 285\\
350  &  177  &  38 580 626 759  &  9 768 511 147\\
352  &  178  &  40 806 395 661  &  10 396 040 696\\
354  &  179  &  42 842 199 753  &  11 037 658 075\\
356  &  180  &  45 278 616 586  &  11 730 538 496\\
358  &  181  &  47 513 679 057  &  12 446 446 419\\
360  &  182  &  50 189 039 868  &  13 221 751 502\\
362  &  183  &  52 628 839 448  &  14 010 515 381\\
364  &  184  &  55 562 506 886  &  14 874 753 568\\
366  &  185  &  58 236 270 451  &  15 754 940 959\\
368  &  186  &  61 437 700 788  &  16 705 334 454\\
370  &  187  &  64 363 670 678  &  17 683 643 273\\
372  &  188  &  67 868 149 215  &  18 744 292 915\\
374  &  189  &  71 052 718 441  &  19 816 289 281\\
376  &  190  &  74 884 539 987  &  20 992 425 825\\
378  &  191  &  78 364 039 771  &  22 186 413 139\\
380  &  192  &  82 532 990 559  &  23 475 079 272\\
382  &  193  &  86 329 680 991  &  24 795 898 388\\
384  &  194  &  90 881 152 117  &  26 227 197 453\\
386  &  195  &  95 001 297 565  &  27 670 862 550\\
388  &  196  &  99 963 147 805  &  29 254 036 711\\
390  &  197  &  104 453 597 992  &  30 852 950 986\\
392  &  198  &  109 837 310 021  &  32 581 366 295\\
394  &  199  &  114 722 988 623  &  34 345 173 894\\
396  &  200  &  120 585 261 143  &  36 259 212 641\\
398  &  201  &  125 873 325 588  &  38 179 777 473\\
400  &  202  &  132 247 999 328  &  40 286 153 024\\
\hline
\end{tabular}
\end{center}
}
\caption{Counts of fullerenes and IPR fullerenes. nv is the number of vertices and nf is the number of faces.}

\label{table:degree_counts_5}

\end{table}

Our generator constructs larger fullerenes from smaller ones, so in
order to generate all fullerenes with $n$ vertices, all fullerenes
with at most $n-4$ vertices have to be generated as well (recall that
an $L_0$ expansion increases the number of vertices by 4).
So generating all fullerenes
with at most $n$ vertices gives only a small overhead compared to
generating all fullerenes with exactly $n$ vertices. In \textit{fullgen} the
overhead is considerably bigger as it does not construct fullerenes from
smaller fullerenes.  For example, \textit{buckygen} can generate all fullerenes with
$n \in [290,300]$ vertices more than 15 times faster than \textit{fullgen}. More comparisons with \textit{fullgen} can be found in Table~\ref{table:fullerene_times}.


\section{Closing remarks}

We have described a new fullerene generator \textit{buckgen} which is
considerably faster than \textit{fullgen}, which is the only previous
generator capable of reaching 100 vertices.
The generation cost is now likely to be lower than that of any 
significant computation performed on the generated structures.

After correction of an error in \textit{fullgen}, we now have two independent
counts of fullerenes up to 380 in full agreement, and values up to 400
vertices from \text{buckgen}.

The latest version of \textit{buckygen} can be downloaded from~\cite{buckygen-site}. \textit{Buckygen} is also part of the \textit{CaGe} software package~\cite{cage}.

\subsection{Acknowledgements}

This work was carried out using the Stevin Supercomputer
Infrastructure at Ghent University.  Jan Goedgebeur is supported by a
PhD grant from the Research Foundation of Flanders (FWO).  Brendan
McKay is supported by the Australian Research Council.


\end{document}